\documentclass[reqno,12pt]{amsart} 
\usepackage{esint,mathrsfs,amsmath,amssymb,amsfonts}
\usepackage{verbatim}

\usepackage{color}
\usepackage[hyperfootnotes=false]{hyperref}
\hypersetup{
  colorlinks,
  citecolor=blue,
  linkcolor=red,
  urlcolor=blue}

\usepackage[left=1.2in, right=1.2in, bottom=1.5in]{geometry}
\newtheorem{theorem}{Theorem}[section]
\newtheorem{lemma}[theorem]{Lemma}
\newtheorem{proposition}[theorem]{Proposition}

\theoremstyle{definition}
\newtheorem{definition}[theorem]{Definition}

\theoremstyle{remark}

\numberwithin{equation}{section}

\newcommand{\R}{\mathbb{R}}

\newcommand{\dist}{\text{dist}}

\begin{document}

\title[The sharp estimate of nodal sets in polytopes]{THE SHARP ESTIMATE OF NODAL SETS \\FOR DIRICHLET LAPLACE EIGENFUNCTIONS IN POLYTOPES}

\author{Yingying Cai}
\address{Yingying Cai: Academy of Mathematics and Systems Science, Chinese Academy of Sciences, Beijing, China.}
\email{caiyingying@amss.ac.cn}

\author{Jinping Zhuge}
\address{Jinping Zhuge: Morningside Center of Mathematics, Academy of Mathematics and systems science,
Chinese Academy of Sciences, Beijing, China.}
\email{jpzhuge@amss.ac.cn}


\maketitle

\begin{abstract}
Let $P$ be a bounded $n$-dimensional Lipschitz polytope, and let $\varphi_{\lambda}$ be a Dirichlet Laplace eigenfunction in $P$ corresponding to the eigenvalue $\lambda$. We show that the $(n-1)$-dimensional Hausdorff measure of the nodal set of $\varphi_{\lambda}$ does not exceed $C(P)\sqrt{\lambda}$. Our result extends the previous ones in quaisconvex domains (including $C^1$ and convex domains) to general polytopes that are not necessarily quasiconvex.
\end{abstract}

\section{Introduction}

Let $\varphi_\lambda$ be the Laplace eigenfunction corresponding to the eigenvalue $\lambda>0$ on an $n$-dimensional compact smooth Riemannian manifold $(\mathcal{M},g)$. Let $Z(\varphi_\lambda) = \{ x\in \mathcal{M}: \varphi_\lambda(x) = 0 \}$ be the nodal set of $\varphi_\lambda$.
S.T. Yau conjectured that $(n-1)$-dimensional Hausdorff measure of $Z(\varphi_\lambda)$ is comparable to $\sqrt{\lambda}$. The original conjecture is still open, while much progress can be found in, e.g., \cite{DF88,DF90,L18a,L18b,LM18,Z19,LZ22,G22,H23}. We refer to \cite{LM20} for a detailed introduction on the Yau's conjecture.

In this paper, we study the nodal sets of Laplace eigenfunctions in a bounded domain $\Omega \subset \R^n$ satisfying the Dirichlet boundary condition on $\partial \Omega$. We are interested in how the regularity or geometry of the boundary affects the size of nodal sets near the boundary. Precisely, let $(\varphi_\lambda,\lambda)$ be the eigenpair satisfying $-\Delta \varphi_\lambda =\lambda \varphi_\lambda$ in $\Omega$ and $\varphi_\lambda = 0$ on $\partial \Omega$. If $\partial \Omega = \emptyset$ or $\partial \Omega$ is real analytic, the Yau's conjecture was proved by Donnelly-Fefferman \cite{DF88,DF90}. We remark that the sharp lower bound estimate, namely $\mathbb{H}^{n-1}(Z(\varphi_\lambda)) \ge c(\Omega) \sqrt{\lambda}$, actually holds independent of the boundary geometry, due to the interior density of nodal sets. However, the boundary geometry will play a crucial role in the upper bound estimate as the nodal set may concentrate near the irregular boundary.
Recently, Logunov-Malinnikova-Nadirashvili-Nazarov \cite{LMNN21} proved the sharp upper bound for $C^1$ domains (or Lipschitz domains with sufficiently small Lipshcitz constant), namely $\mathbb{H}^{n-1}(Z(\varphi_\lambda)) \le C(\Omega) \sqrt{\lambda}$. Following the similar idea and taking advantage of the properties of convex domains, Zhu-Zhuge \cite{ZZ23} generalized the result to quasiconvex domains, which is a class of Lipschitz domains that includes both 
$C^1$ and convex domains. As far as we know, all the previous results do not cover (nonconvex) polygonal domains with very simple geometry. Note that the upper bound estimate of nodal sets is a global property and the previous results cannot rule out the possibility of concentration of the nodal sets in the vicinity of nonconvex corners or ridges of polygonal domains.

The contribution of this paper is to prove the sharp estimate of the nodal sets for Dirichlet Laplace eigenfunctions in bounded Lipschitz polytopes, which are higher-dimensional generalization of two-dimensional polygons and three-dimensional polyhedras. The precise mathematical characterization of polytopes will be given in Section \ref{sec.Polytopes}. The main result of this paper is stated as follows.

\begin{theorem}\label{thm:1} 
Let P be a bounded n-dimensional Lipschitz polytope. Let $\varphi_{\lambda}$ be a Dirichlet eigenfunction of the Laplace operator corresponding to the eigenvalue $\lambda$, i.e.,
\begin{equation*}
\left\{\begin{aligned}
-\Delta\varphi_\lambda & = \lambda\varphi_{\lambda} \quad & \text{in }& P, \\
\varphi_\lambda & =0\quad & \text{on }& \partial P.
\end{aligned}\right.
\end{equation*}
Then, the $(n-1)$-dimensional Hausdorff measure of the nodal set of $\varphi_{\lambda}$, denoted by $\mathbb{H}^{n-1}(Z(\varphi_{\lambda}))$, satisfies
\begin{equation}
\mathbb{H}^{n-1}\left(Z\left(\varphi_{\lambda}\right)\right)\le C\sqrt{\lambda},
\end{equation}
where $C$ depends only on $P$ and $n$.
\end{theorem}

The proof of Theorem \ref{thm:1} relies on a lifting argument to convert the eigenfunctions in $P$ to harmonic functions in $P\times \R$. Then the difficulty of estimating the nodal sets for harmonic functions lies in the uniform control of the doubling index centered at any points, particularly when the points are close to the corners or ridges. We develop an approach to control the doubling index which involves the monotonicity and propagation of the doubling index in star-shaped domains, the geometric properties of polytopes (near corners, ridges, etc) and the estimation of maximum star-shape radius (MSR). We remark that our approach may also apply to general second order elliptic equations in piecewise smooth/convex domains. To avoid complexity, we do not pursue this generality in the present paper.

\textbf{Notations.} Throughout the paper, we denote by $c,c_1, c_2,\cdots,$ small positive constants and $C,C_1,C_0, \cdots,$ large positive constants depending only on $n$ and $P$, which may change from line to line.
We use $B_r(x)$ or $B(x,r)$ to denote the ball of radius $r$ centered at $x$. Denote by $d(A,B)$ the Euclidean distance between the objects $A$ and $B$ (which could be points or sets). Denote by $B(E,r) = \{x: d(x,E) < r \}$ the $r$-neighborhood of the set $E$. We write $a\wedge b = \min \{a,b \}$ for $a,b\in \R$.

\textbf{Organizations.} The remainder of the paper is organized as follows. In Section \ref{sec.DI} we recall the monotonicity of the doubling index in star-shaped domains (the outline of the proof will be given in Appendix) and prove a property on the propagation of doubling index. In Section \ref{sec.Polytopes}, we study the useful properties of polytopes. Section \ref{sec.Unif} is dedicated to the uniform estimate of the doubling index. The main theorem is proved in Section \ref{sec.nodal}.
 
\textbf{Acknowledgements.} J.Z. is partially supported by grants for Excellent Youth from the NSFC and AMSS-CAS.
 Y.C. would like to extend her deepest gratitude to Prof. Liqun Zhang for his careful guidance.

\section{Monotonicity of doubling index}\label{sec.DI}

An important tool to study the nodal sets or vainshing order of the eigenfunctions of elliptic operators is the Almgren's frequency function or doubling index. 
In this section, we recall the monotonicity of doubling index in star-shaped Lipschitz domains.

\begin{definition}\label{def.Lip}
We say a domain $\Omega$ is Lipschitz, if there exits $r_{0}>0$ such that for every $x_{0}\in \partial\Omega$, the boundary patch $\partial\Omega \cap B_{r_{0}}(x_{0})$, after a rigid transformation, can be expressed as a Lipschitz graph $x^{n}=\phi(x')$ such that 
\begin{equation}\label{LipGraph}
B_{r_{0}}(x_{0})\cap \Omega =B_{r_{0}}(x_{0})\cap \left \{ x=(x^1,x^2,\dots ,x^n)=(x',x^n):x^n > \phi (x') \right \} 
\end{equation}
and we call $r_{0}$ the Lipschitz radius of $\Omega$. Moreover, the Lipschitz constants of these graphs are uniformly bounded by some constant $L$ . 
\end{definition}

\begin{definition}
    We say a connected Lipschitz domain $\Omega$ is star-shaped with respect to some $x\in \overline{\Omega}$ if for almost every $y\in \partial \Omega$,
\begin{equation}
    (y-x) \cdot n(y) \ge 0,
\end{equation}
where $n(y)$ is the outer normal to $\partial \Omega$ at $y$.
Alternatively, a Lipschitz domain $\Omega$ is star-shaped with respect to $x\in \overline{\Omega}$ if for each $y\in \partial \Omega$, the line segment connecting $x$ and $y$ is entirely contained in $\overline{\Omega}$.
\end{definition}

Let $u$ be a harmonic function in $\Omega$.
Define the doubling index of $u$ in $\Omega$ centered at $x \in \overline{\Omega}$ by
\begin{equation*}
N_{u}(x, r)=\log_2 \frac{\int_{B(x,2r)\cap\Omega} u^{2}}{\int_{B(x,r)\cap \Omega} u^{2}}.
\end{equation*}

The following lemma was essentially proved in \cite[Lemma 3.1]{KN98}. We outline the proof in Appendix.

\begin{lemma}\label{lem.monotonicityDL} 
Let $\Omega$ be a Lipschitz domain in $\mathbb{R}^{n}$. Let $x \in \overline{\Omega}$ and $R>0$. Assume that $\Omega \cap B(x,2R )$ is star-shaped with respect to $x$. Suppose $u\in C(\overline{\Omega})$ is a non-zero harmonic function in $\Omega \cap B(x,2R )$ and $u=0$ on $\partial \Omega \cap B(x,2R).$  Then for any $0<s<r<R$, we have
\begin{equation}\label{est.monotonicityDL}
N_{u}(x,s)\le N_{u}(x,r).
\end{equation}
\end{lemma}

We also need a lemma to quantify the propagation of doubling index in star-shaped domains when the center is shifted.

\begin{lemma}\label{lem.PropagationOfDB} Let $\Omega$ be an $n$-dimensional Lipschitz domain and $0\in \partial \Omega$. Let $u$ be a harmonic function in $\Omega \cap B_R(0)$ and $u = 0$ on $\partial \Omega \cap B_R(0)$. Let $x\in \overline{\Omega}\cap B_\frac{R}{2}(0)$. Suppose that there is an embedded Lipschitz curve 
 $\gamma :[0,1] \to \overline{\Omega} \cap B(0,R)$ satisfying the following conditions:
 \begin{itemize}
    \item[(i)] $\gamma(0) = 0$ and $\gamma(1) = x$.

    \item[(ii)] For some $0<r<\frac14 R$ and all $s\in [0,1]$, $\Omega \cap B(\gamma(s), 4r)$ is contained in $\Omega \cap B_{R\wedge C_1 r}(0)$ for some $C_1 > 1$ and is star-shaped with respect to $\gamma(s)$.
    
     \item[(iii)] For some constant $C_2>1$,  $|\gamma| \le C_2 r$.

    \item[(iv)] $\Omega \cap B_R(0)$ is star-shaped with respect to $0$.
 \end{itemize}
Then there exists $C=C(C_1, C_2)$ such that
 \begin{equation}\label{est.propagation}
N_{u}(x,r)\le C N_{u}(0,\frac12 R).
\end{equation}
\end{lemma}

\begin{proof}
    Let $R_1 := R \wedge C_1 r$. Without loss of generality, we assume
    \begin{equation}\label{eq.assump.BR1}
        \int_{\Omega \cap B_{R_1}(0)} u^2 = 1.
    \end{equation}
    Let $N_0 = N_u(0,R_1/2)$. Since $\Omega \cap B_R(0)$ is star-shaped with respect to the origin, by the monotonicity of the doubling index, we have $N_0 \le N_{u}(0,\frac12 R)$.
    
    By the monotonicity of the doubling index again and the assumption \eqref{eq.assump.BR1}, we can show
    \begin{equation}\label{est.PBr0}
        \int_{\Omega \cap B_r(0)} u^2 \ge \exp_2(- N_0 (1+ \log_2 \frac{R_1}{r})) \ge \exp_2(- N_0 (1+ \log_2 C_1)),
    \end{equation}
    where $\exp_2(a) = 2^a.$
    Let $x_1 = \gamma(s_1)$ for some $s_1\in (0,1)$ and $|x_1 - 0| = r$. Since $B_{4r}(x_1)$ is star-shaped with respect to $x_1$ (due to (iii)), using the monotonicity of the doubling index (or the three ball inequality), we have
    \begin{equation}
    \begin{aligned}
        \int_{B_r(0)\cap \Omega} u^2 &\le \int_{B_{2r}(x_1)\cap \Omega} u^2 \\
        &\le \bigg( \int_{B_r(x_1)\cap \Omega} u^2 \bigg)^{1/2} \bigg( \int_{B_{4r}(x_1)\cap \Omega} u^2 \bigg)^{1/2} \le \bigg( \int_{B_r(x_1)\cap \Omega} u^2 \bigg)^{1/2},
    \end{aligned}
    \end{equation}
    where we also used the assumptions $B_{4r}(x_1)\cap \Omega \subset B_{R_1}(0) \cap \Omega$ and $\int_{\Omega\cap B_{R_1}(0)} u^2 = 1$. It follows that
    \begin{equation}\label{est.PropagSmall}
        \int_{B_r(x_1)\cap \Omega} u^2 \ge \bigg( \int_{B_r(0)\cap \Omega} u^2 \bigg)^2.
    \end{equation}

    Now we repeat this process and find a sequence of $x_j$ connecting $0$ and $x$ with $1\le j\le m$ and $x_m = x$ such that $|x_j - x_{j-1}| = r$ for $j<m$ and $|x_m - x_{m-1}| \le r$. By (ii), we see that $m\le |\gamma|/r \le C_2$. The previous argument yields
    \begin{equation}
        \int_{B_r(x_j)\cap \Omega} u^2 \ge \bigg( \int_{B_r(x_{j-1})\cap \Omega} u^2 \bigg)^2, \qquad \text{for any } 2\le j\le m.
    \end{equation}
    This together with \eqref{est.PropagSmall} and \eqref{est.PBr0} gives
    \begin{equation}
        \int_{B_r(x)\cap \Omega} u^2 \ge \bigg( \int_{B_r(0)\cap \Omega} u^2 \bigg)^{2^m} \ge \exp_2(- N_0 (1+ \log_2 C_1)2^{C_2} ).
    \end{equation}
    Consequently, by the last estimate and \eqref{eq.assump.BR1} (note the fact $B_{2r}(x) \cap \Omega \subset B_{R_1}(0) \cap \Omega$ ), we have
    \begin{equation}
        N_u(x,r) =  \log_2 \frac{\int_{B_{2r}(x)\cap \Omega} u^2 }{\int_{B_r(x)\cap \Omega} u^2} \le N_0 (1+ \log_2 C_1)2^{C_2}  \le C  N_{u}(0,\frac12 R).
    \end{equation}
    This ends the proof.
\end{proof}

\section{Polytopes}\label{sec.Polytopes}
\subsection{Geometry of polytopes}
A polytope is a geometric object with flat sides (faces), generalizing the two dimensional polygons and three-dimensional polyhedras to any number of dimensions. An $n$-dimensional polytope (as a connected domain) in $\R^n$, bounded by a finite number of $(n-1)$-dimensional facets, will be called an $n$-polytope. These facets are themselves $(n-1)$-polytopes, whose facets are $(n-2)$-polytopes (called ridges), and so forth. These bounding $k$-polytopes (with $0\le k< n$) will be called $k$-faces of the original $n$-polytope. Typically, $k$-faces arise as the intersection of $(k+1)$-faces. For example, 0-faces are vertices, 1-faces are edges of line segments and $(n-1)$-faces are facets of the original $n$-polytope.

Let $P$ be a given $n$-polytope. For each integer $j$ with $1\le j\le n$, we define $\mathcal{F}^{n-j}$ as the collection of all $(n-j)$-faces of $P$ and $\mathscr{F}^{n-j}$ as the union of all $(n-j)$-faces of $P$. Let $\mathcal{F} = \cup_{j=1}^n \mathcal{F}^{n-j}$. By the previous definition, for each $2 \le j\le n$, we have
\begin{equation*}
\mathscr{F}^{n-j}\subset \mathscr{F}^{n-j+1},\quad \mathscr{F}^{n-1}=\partial P.
\end{equation*}
Moreover, for a given $(n-j)$ face $F \in \mathcal{F}^{n-j}$, we denote by $\partial F$ the boundary of $F$, which consists of $(n-j-1)$-faces contained in $\mathscr{F}^{n-j-1}$.

The following is a Euclidean geometric fact of polytopes that characterizes the distance between points on a given $(n-j)$-face and other $(n-1)$-faces not containing the given face.


\begin{lemma}\label{lem.distF}
Let P be a $n$-polytope. Then there exists $c^{*}>0$ such that for any $F \in \mathcal{F}$ and $G \in \mathcal{F}^{n-1}$ satisfying $F\cap G \neq F$ (or equivalently $F \nsubseteq G$), we have
\begin{equation}\label{eq.dxG}
d(x,G)\ge c^{*}d(x,\partial F),\quad \text{for any } x\in F.
\end{equation}
\end{lemma}

\begin{proof}
We consider the cases $F\cap G\neq \emptyset$ and $F\cap G = \emptyset$ separately.
When $F\cap G\ne \emptyset$, we have $F\cap G  = \partial F \cap G$. A simple geometric observation yields the existence of a constant $c_{F,G} \in (0,1]$ such that for any $x\in F$
\begin{equation}
d(x,G)\ge c_{F,G}d(x, F\cap G)\ge c_{F,G}d(x, \partial F).
\end{equation}
On the other hand, we note that the number of faces of $P$ is finite. Hence,
\begin{equation}
c':=\min \{ c_{F,G}: F\in \mathcal{F}, G\in \mathcal{F}^{n-1}, F\cap G \neq F, F\cap G\ne \emptyset \} > 0.
\end{equation}
Thus, we have $d(x, G)\ge c' d(x,\partial F)$ in this case.

When $F\cap G=\emptyset$, we have $d(F,G) > 0$. Since the number of the pairs under this condition is finite, we can define 
\begin{equation}
c'' :=\min \{ d(F,G): F\in \mathcal{F}, G\in \mathcal{F}^{n-1},F\cap G= \emptyset\} >0.
\end{equation}
Consequently, for any $x\in F$, we have 
\begin{equation}
d(x,G)\ge d(F,G)\ge c'' \ge \frac{c''}{\text{diam}(P)}d(x, \partial F).\end{equation}
Therefore, combining the above two cases and choosing $c^{*}=c' \wedge \frac{c''}{diam(P)}$, we establish (\ref{eq.dxG}).
\end{proof}

\subsection{Maximum star-shape radius (MSR)}
The key to estimate the nodal sets is to control the doubling index centered at every point in $\overline{P}$. Due to the monotonicity of doubling index in star-shaped regions, it is natural to introduce the maximum star-shape radius at each point. Precisely, define
\begin{equation}
\begin{aligned}
    R^*(x) = \max \{ r\ge 0 : & \text{the connected component of } B_r(x) \cap P \\
    & \text{ is star-shaped with respect to } x  \} \wedge \text{diam} (P).
\end{aligned}
\end{equation}
The following crucial lemma gives lower bounds of MSR, based essentially on the geometry of polytopes. In particular, it tells us that the MSR can be large on the interior portions of any $j$-faces.
\begin{proposition}\label{prop.lowerbound.Rx}
    For each point $x\in \overline{P}$, we have
    \begin{equation}\label{est.R*x}
        R^*(x) \ge d(x, \mathscr{F}^{n-2}).
    \end{equation}
    In addition, if $x \in \mathscr{F}^{n-i}\setminus \mathscr{F}^{n-i-1}$ with $i=2,3,\cdots,n-1.$, then
    \begin{equation}\label{est.R*x2}
        R^*(x)\ge c^* d (x,\mathscr{F}^{n-i-1}),
\end{equation}
where $c^*>0$ is the constant given in Lemma \ref{lem.distF}.
    Moreover, if $x\in \mathscr{F}^0$, then
    \begin{equation}\label{est.R*x3}
        R^*(x) \ge \inf_{G\in \mathcal{F}^{n-1}, x\notin G}  d(x, G) \ge R_0 > 0,
    \end{equation}
    for some constant $R_0$ depending on $P$.
\end{proposition}

\begin{proof}
    Let $r_x = \text{dist}(x, \mathscr{F}^{n-2})$. To show \eqref{est.R*x}, it suffices to show that the connected component of $B_{r_x}(x) \cap P$ is star-shaped. In fact, by the definition of $r_x$, we see that $B_{r_x}(x) \cap P$ does not intersect with $\mathscr{F}^{n-2}$ and thus the boundary of $B_{r_x}(x) \cap P$ consists of a portion of $\partial B_{r_x}(x)$ and flat faces (interior portions of $(n-1)$-faces). Hence, $B_{r_x}(x) \cap P$ is convex and thus star-shaped.

    Now we consider the particular case $x \in \mathscr{F}^{n-i}\setminus \mathscr{F}^{n-i-1}$ and prove \eqref{est.R*x2}. Recall that $\mathcal{F}^{n-i}$ is the collection of all $(n-i)$-faces. Let $\mathcal{F}^{n-i} = \{ F^{n-i}_{j}: 1\le j\le N_{n-i} \}$, where $N_{n-i}$ is the number of all $(n-i)$-faces. Since $x\notin \mathscr{F}^{n-i-1}$, $x$ is contained in exactly one $(n-i)$-face. Without loss of generality, assume $x\in F^{n-i}_1$. Moreover, due to the fact $\partial F^{n-i}_1 \subset \mathscr{F}^{n-i-1}$,
    \begin{equation}
        d(x, \partial F^{n-i}_1) \ge d(x,\mathscr{F}^{n-i-1}).
    \end{equation}
    Now, for any $(n-1)$-face $G \in \mathcal{F}^{n-1}$ not containing $F^{n-i}_1$, by Lemma \ref{lem.distF},
    \begin{equation}\label{est.dxG}
        \dist(x, G) \ge c^* d(x, \partial F^{n-i}_1) \ge c^* d(x,\mathscr{F}^{n-i-1}).
    \end{equation}
    Let $r_x^{n-i-1}:=d (x,\mathscr{F}^{n-i-1})$. Hence, \eqref{est.dxG} implies that $B(x, c^* r_x^{n-i-1})$ does not intersect with any $(n-1)$-faces not containing $F^{n-i}_1$. This further implies
    \begin{equation}\label{eq.BpartialP}
        B(x, c^* r_x^{n-i-1}) \cap \partial P =B(x, c^* r_x^{n-i-1}) \cap  \bigcup_{F\in \mathcal{F}^{n-1}, x\in F} F.
    \end{equation}
    
    Next we claim that $B(x, c^* r_x^{n-i-1}) \cap P$ is star-shaped with respect to $x$. Notice that
    \begin{equation}\label{eq.LocalBdry}
        \partial (B(x, c^* r_x^{n-i-1}) \cap P) = \Big (\partial B(x, c^* r_x^{n-i-1}) \cap P \Big) \bigcup \Big( B(x, c^* r_x^{n-i-1}) \cap \partial P \Big).
    \end{equation}
    It suffices to show that for any $y\in \partial (B(x, c^* r_x^{n-i-1}) \cap P)$, we have $n(y)\cdot (y-x) \ge 0$, where $n(y)$ is the outward normal of $B(x, c^* r_x^{n-i-1}) \cap P$ at $y$. In view of \eqref{eq.LocalBdry}, we have two cases. If $y \in \partial B(x, c^* r_x^{n-i-1}) \cap P$, then $n(y)$ has the same direction as $y-x$ and thus $n(y)\cdot (y-x) > 0$. If $y \in B(x, c^* r_x^{n-i-1}) \cap \partial P$, then by \eqref{eq.BpartialP}, $y\in B(x, c^* r_x^{n-i-1}) \cap F$ for some $F\in \mathcal{F}^{n-1}$ and $x\in F$. Since $F$ is a portion of a hyperplane, $y-x$ is contained in the hyperplane and $n(y)$ (exists a.e.) is perpendicular to the hyperplane. Thus $n(y)\cdot (y-x) = 0$. Thus the two cases combined lead to the claim.

    The above claim then implies
    \begin{equation}
        R^*(x) \ge c^*r_x^{n-i-1} = c^* d (x,\mathscr{F}^{n-i-1}),
    \end{equation}
    as desired.

    Finally, if $x\in \mathscr{F}^0$ and $r^0_x = \inf_{G\in \mathcal{F}^{n-1}, x\notin G}  d(x, G) > 0$, then
    \begin{equation}
        B(x, r^0_x) \cap \partial P = B(x,r^0_x) \cap \bigcup_{F\in \mathcal{F}^{n-1}, x\in F} F.
    \end{equation}
    Following the same argument as before, we can show that $B(x,r^0_x) \cap P$ is star-shaped and $R^*(x) \ge r^0_x$. Since $x\in \mathscr{F}^0$ is a vertex of $P$ and $P$ has a finite number of vertices, then $r^0_x \ge R_0$ for some constant $R_0>0$.
\end{proof}

The next lemma shows that the boundary of a bounded $n$-polytope can be covered by a finite number of star-shaped subdomains.

\begin{proposition}\label{prop.partition}

Let $P$ be a bounded $n$-polytope and $R_0$ given as Proposition \ref{prop.lowerbound.Rx}. Then for a fixed $r_0 \le R_0$, there exist finite number of balls $\{ B(x_{k,j},(\frac{1}{32}c^{*})^{k} c^* r_0 ):j\in \{1,2,\cdots, M_k \} \}_{ \{k=0,1,\cdots, n-1\} }$ with finite overlaps satisfying
\begin{itemize}
    \item[(i)] $x_{0,j} \in \mathscr{F}^0$ and $x_{k,j} \in \mathscr{F}^{k} \setminus \mathscr{F}^{k-1}$ for all $1\le k\le n-1$. 

    \item[(ii)] For each $k$ and $j$, $P \cap B(x_{k,j}, (\frac{1}{32}c^{*})^{k} c^* r_0)$ is star-shaped with respect to $x_{k,j}$.

    \item[(iii)]
    \begin{equation}
        \partial P \subset \bigcup_{k=0}^{n-1} \bigcup_{j=1}^{M_k} B(x_{k,j}, 2(\frac{1}{32}c^{*})^{k+1}r_0).
    \end{equation}
\end{itemize}
\end{proposition}

\begin{proof}
    Let $x_{0,j}$ ($1\le j\le M_0$) be the vertices in $\mathscr{F}^0$. By \eqref{est.R*x3}, $B(x_{0,j}, c^* r_0) \cap P$ is star-shaped with respect to $x_{0,j}$ for each $j$. Let 
    \begin{equation}\label{eq.Fint1}
        \mathscr{F}^1_{\rm int} = \mathscr{F}^1 \setminus B(\mathscr{F}^0, \frac{1}{32}c^* r_0).
    \end{equation}
    Since $\mathscr{F}^1_{\rm int}$ is compact, we can find a finite number of balls $B(x_{1,j}, \frac{1}{32} (c^*)^2 r_0)$ (with finite overlaps) centered on $\mathscr{F}^1_{\rm int} \subset \mathscr{F}^1 \setminus \mathscr{F}^0$ such that
    \begin{equation}\label{eq.CoverFint1}
        \mathscr{F}^1_{\rm int} \subset \bigcup_{j=1}^{M_1} B(x_{1,j},  (\frac{1}{32} c^*)^2 r_0).
    \end{equation}
    Combining \eqref{eq.Fint1} and \eqref{eq.CoverFint1}, we have
    \begin{equation}
        \mathscr{F}^1 \subset B(\mathscr{F}^1,(\frac{1}{32} c^*)^2 r_0 ) \subset \bigcup_{k=0}^1 \bigcup_{j=1}^{M_k} B(x_{k,j}, 2(\frac{1}{32} c^*)^{k+1} r_0).
    \end{equation}
    We claim that $B(x_{1,j}, \frac{1}{32}(c^*)^2 r_0) \cap P$ is star-shaped with respect to $x_{1,j}$. Indeed, by \eqref{eq.Fint1}, $\text{dist}(x_{1,j}, \mathscr{F}^0) \ge \frac{1}{32}c^* r_0$. Then by \eqref{est.R*x2}, we have $R^*(x_{1,j}) \ge \frac{1}{32}(c^*)^2 r_0$. This means that $B(x_{1,j}, \frac{1}{32}(c^*)^2 r_0) \cap P$ is star-shaped with respect to $x_{1,j}$. Hence, we have proved (i) and (ii) for $k=1$.

    Now consider $k=2$. Let
    \begin{equation}
        \mathscr{F}^2_{\rm int} = \mathscr{F}^2 \setminus B(\mathscr{F}^1, (\frac{1}{32} c^*)^2 r_0).
    \end{equation}
    Since $\mathscr{F}^2_{\rm int}$ is compact, we can find finite number of balls $B(x_{2,j}, (\frac{1}{32} c^*)^2 c^* r_0)$ centered on $\mathscr{F}^2_{\rm int} \subset \mathscr{F}^2 \setminus \mathscr{F}^1$ such that
    \begin{equation}
        \mathscr{F}^2_{\rm int} \subset \bigcup_{j=1}^{M_2} B(x_{2,j}, (\frac{1}{32} c^*)^3 r_0).
    \end{equation}
    Then, doubling the radius of the balls, we have
    \begin{equation}
        \mathscr{F}^2 \subset B(\mathscr{F}^2, (\frac{1}{32} c^*)^3 r_0) \subset \bigcup_{k=0}^2 \bigcup_{j=1}^{M_k} B(x_{k,j}, 2(\frac{1}{32} c^*)^{k+1} r_0).
    \end{equation}
    By the same reason as before, we can show that $B(x_{2,j}, (\frac{1}{32} c^*)^2 c^* r_0) \cap P$ is star-shaped with respect to $x_{2,j}$.

    Repeating this process, we find a sequence of balls $B(x_{k,j}, (\frac{1}{32}c^*)^k c^* r_0)$ with $x_{k,j} \in \mathscr{F}^k \setminus \mathscr{F}^{k-1}$ such that $B(x_{k,j}, (\frac{1}{32}c^*)^k c^* r_0) \cap P$ is star-shaped with respect to $x_{k,j}$ and
    \begin{equation}\label{est.coverDP}
        \mathscr{F}^i \subset B(\mathscr{F}^i, (\frac{1}{32} c^*)^{i+1} r_0) \subset \bigcup_{k=0}^i \bigcup_{j=1}^{M_k} B(x_{k,j}, 2(\frac{1}{32} c^*)^{k+1} r_0).
    \end{equation}
    for every $i=1,2,\cdots, n-1$. The proves (i) and (ii) for every $k\le n-1$, while (iii) follows from \eqref{est.coverDP} with $i = n-1$ and the fact $\partial P = \mathscr{F}^{n-1}$.
\end{proof}

The following lemma shows that the MSR does not decrease rapidly if the center moves nontangentially from a boundary point to the interior points.
\begin{lemma}\label{lem.CofDI}
    Let $B_{r_{0}/2}(0) \cap P$ be given as in \eqref{LipGraph} (up to a rigid transformation)
    and $x\in B_{r_{0}/2}(0) \cap P$. Let $R^{**}(x):= R^*(x) \wedge \frac14 r_0$. Then for any $t\in (0,\frac12 R^{**}(x)]$, we have
    $R^*(x+te_n) \ge \frac12 R^{**}(x)$.
\end{lemma}
\begin{proof}
    Since $R^{**}(x) \le R^*(x)$, $B(x, R^{**}(x)) \cap P$ is star-shaped with respect to $x$. This implies for any $y\in B(x, R^{**}(x)) \cap \partial P$, 
    \begin{equation}
        n(y)\cdot (y-x) \ge 0.
    \end{equation}
    Note that for any $t\in (0,\frac12 R^{**}(x)]$, $B(x+te_n, \frac12 R^{**}(x)) \subset B(x, R^{**}(x))$. Thus, for any $y\in B(x+te_n, \frac12 R^{**}(x)) \cap \partial P$,
    \begin{equation}
        n(y) \cdot (y - (x+te_n)) = n(y)\cdot (y-x) - t n(y)\cdot e_n \ge - t n(y)\cdot e_n \ge 0,
    \end{equation}
    where the last estimate uses the fact $n(y)\cdot e_n \le 0$. This implies that $B(x+te_n, \frac12 R^{**}(x))$ is star-shaped with respect to $x+te_n$ and thus $R^*(x+te_n) \ge \frac12 R^{**}(x)$.
\end{proof}


\section{Uniform bound of doubling index}\label{sec.Unif}
In this section, we prove a crucial lemma that is used to control the doubling index in a star-shaped subdomain of a polytope.

\begin{lemma}\label{lem.unifDI}
    Assume that $\Omega = B_{r_0}(0) \cap P$ is star-shaped with respect to $0$. Let $u$ be a harmonic function in $\Omega$ with $u = 0$ on $B_{r_0}(0) \cap \partial P$. Then there exist $c_1 = c_1(L)>0$ such that for any $x\in B_{r_0/8}(0) \cap P$ and any $0<r< c_1 \{ R^*(x) \wedge \frac14 r_0 \}$,
    \begin{equation}\label{est.unifdoubling}
        N_u(x, r) \le C N_u(0, \frac12 r_0).
    \end{equation}
\end{lemma}
\begin{proof}
    Without loss of generality, we assume that $B_{r_0}(0) \cap P$ is given as in \eqref{LipGraph} above a Lipschitz graph.
    Let $c \le \frac14$ be a small number. Fix $x\in B_{r_0/8}(0) \cap P$. If $ c r_0 \le R^{**}(x) := R^*(x) \wedge \frac14 r_0$, then by Lemma \ref{lem.CofDI}, for any $0<t\le \frac12 R^{**}(x)$,
    \begin{equation}
        R^*(x+ t e_n) \ge \frac12 R^{**}(x).
    \end{equation}
    Since $B_{r_0}(0) \cap P$ is star-shaped, then by Lemma \ref{lem.CofDI} again, for any $0<t\le \frac12 r_0$,
    \begin{equation}
        R^*(0+t e_n) \ge \frac12 r_0.
    \end{equation}
    Recall that $B_{r_0}(0) \cap P$ is star-shaped with respect to the origin, then the line segment connecting $x$ and $0$ is contained in $B_{r_0}(0) \cap P$ (above the boundary graph). As a simple geometric observation, the line segment connecting $x+\frac12 R^{**}(x) e_n$ to $0+\frac12 R^{**}(x) e_n$ has a distance at least $\frac12 \ell_0 R^{**}(x)$ away from the boundary, where
    \begin{equation}
        \ell_0 = \frac{1}{\sqrt{1+L^2}}.
    \end{equation}
    Combining this line segment with the vertical line segments from $x$ to $x+\frac12 R^{**}(x) e_n$ and from $0+\frac12 R^{**}(x) e_n$ to $0$, we obtain a Lipschitz curve connecting $x$ and $0 $. Moreover, the Lipschitz curve is contained in $B_{r_0}(0) \cap P$ and $B_{\frac12 \ell_0 R^{**}(x)}(y) \cap P$ is star-shaped for any $y$ on the Lipschitz curve. Also the length of the curve is bounded by $Cr_0 \le C_1 R^{**}(x)$.
    Then by Lemma \ref{lem.PropagationOfDB}, we have
    \begin{equation}
        N_u(x, \frac18 \ell_0 R^{**}(x)) \le C N_u(0, \frac12 r_0),
    \end{equation}
    where $C$ depends on $c$ and $L$. If $c r_0 \le R^*(x) \le \frac14 r_0$, then the above argument also implies $N_u(x, \frac18 \ell_0 R^*(x)) \le C N_u(0, \frac12 r_0)$. This typical argument will be applied repeatedly below.

    Now we consider the case $R^*(x) < c r_0$ with sufficiently small $c >0$.

    Let $\pi(x) \in \partial P$ be the vertical projection of $x$ onto $\partial P$. Thus $x = \pi(x) + e_n |\pi(x) - x|$. We set $z_{n-1} = \pi(x)$. Let $z_{n-2}$ be the closest point  to $z_{n-1}$ from $\mathscr{F}^{n-2} \cap B_{r_0/2}(0)$.
    In general, let $z_{n-j}$ be the closest point to $z_{n-j+1}$ from $\mathscr{F}^{n-j} \cap B_{r_0/2}(0)$, until either one of the following two cases happens.
    
    Case 1: for some $2\le j_0\le n$, $z_{n-j_0}$ does not exist in $\mathscr{F}^{n-j_0} \cap B_{r_0/2}(0)$, i.e., $\mathscr{F}^{n-j_0} \cap B_{r_0/2}(0) =\emptyset$.
    
    Case 2: all $z_{n-j}$ exist for $2\le j\le n$.
    
    Clearly, if Case 1 does not happen, then Case 2 must happen. 
    
    For case 1, we claim that there exists $c' >0$ and some $j$ with $2\le j< j_0$ such that $R^*(z_{n-j}) \ge c' r_0$.   
    Actually, if $z_{n-j_0+1} \in B_{r_0/4}(0)$, then $R^*(z_{n-j_0+1}) > \frac14 c^* r_0$ due to \eqref{est.R*x2} and the fact that $z_{n-j_0}$ does not exist in $B_{r_0/2}(0)$. Now we consider the situation $z_{n-j_0+1} \notin B_{r_0/4}(0)$, then
    \begin{equation}
        \frac18 r_0 < |z_{n-1} - z_{n-j_0+1}| \le \sum_{j=1}^{j_0-2} |z_{n-j} - z_{n-j-1}|.
    \end{equation}
    The pigeonhole principle implies that there exists at least one $j$ with $1\le j\le j_0-2$ such that
    \begin{equation}
        |z_{n-j} - z_{n-j-1}| \ge \frac{1}{8n} r_0.
    \end{equation}
    Consequently, by \eqref{est.R*x2}, we have $R^*(z_{n-j}) \ge c' r_0$.

    For Case 2, thanks to \eqref{est.R*x3}, we also have $R^*(z_0) \ge c' r_0$ for some possibly different $c'>0$. Combining both cases, without loss of generality, we may assume that the integer $j_s$ with $1\le j_s \le j_0$ is the smallest integer such that $R^*(z_{n-j_s}) \ge c' r_0$.

    Next, we will develop a procedure to bound the doubling index through the sequence of points $z_{n-j}$ with a stopping time before reaching $j = j_s$. Let $1\le k_1\le j_s$ be the smallest integer such that
    \begin{equation}\label{cond.k1}
        R^*(z_{n-k_1}) > 32\ell_0^{-1} |z_{n-k_1} - x| \quad \text{and} \quad R^*(z_{n-k_1}) > 32 \ell_0^{-1} R^*(x).
    \end{equation}
    We temporarily assume that such $k_1$ exists.
    Then we show that there exists $C>0$ such that
    \begin{equation}\label{est.N.n-k1}
        N(x, \frac18 \ell_0 R^*(x)) \le CN(z_{n-k_1}, \frac18 \ell_0 R^*(z_{n-k_1})).
    \end{equation}
    In fact, by the selection of $k_1$ and \eqref{est.R*x2}, for each $1\le j < k_1$
    \begin{equation}
    \begin{aligned}
        c^*|z_{n-j-1} - x| - c^*|z_{n-j} - x| & \le c^*|z_{n-j} - z_{n-j-1}| \\
        & \le R^*(z_{n-j}) \\
        & \le 32\ell_0^{-1} |z_{n-j} - x| + 32\ell_0^{-1} R^*(x).
    \end{aligned}
    \end{equation}
    Thus
    \begin{equation}
        |z_{n-j-1} - x| \le \frac{32\ell_0^{-1}+c^*}{c^*} |z_{n-j} - x| + \frac{32 \ell_0^{-1}}{c^*} R^*(x).
    \end{equation}
    It follows by an iteration that
    \begin{equation}\label{est.znk1-x}
        |z_{n-k_1}-x| \le C(n,L) R^*(x).
    \end{equation}
    By this and the assumption \eqref{cond.k1}, and a similar argument as before, we can find a Lipschitz curve contained in $B(z_{n-k_1}, |x-z_{n-k_1}| + R^*(x)) \cap P$ connecting $x$ and $z_{n-k_1}$ whose length is comparable to $|x - z_{n-k_1}| + R^*(x) \lesssim R^*(x)$ (by \eqref{est.znk1-x}). Moreover, $P\cap B(y, \frac12 \ell_0 R^*(x))$ is star-shaped for all $y$ on the Lipschitz curve. Thus \eqref{est.N.n-k1} follows from Lemma \ref{lem.PropagationOfDB}, provided that $k_1$ exists. We point out that the particular constant $32\ell_0^{-1}$ in \eqref{cond.k1} guarantees that the condition (iii) in Lemma \ref{lem.PropagationOfDB} is satisfied with $R = \frac14 \ell_0 R^*(z_{n-k_1})$ and $r = \frac18 \ell_0 R^*(x)$.

    On the other hand, if $k_1$ does not exist, then we can show $R^*(x) \ge c r_0$ for some small $c>0$. Actually, if $k_1$ does not exist, then for any $1 \le j\le j_s$, we have
    \begin{equation}\label{cond.Rz}
        R^*(z_{n-j}) \le 32\ell_0^{-1} |z_{n-j} - x| \quad \text{or} \quad R^*(z_{n-j}) \le 32 \ell_0^{-1} R^*(x).
    \end{equation}
    Consider $j=1$. A simple geometric observation implies
    \begin{equation}\label{est.Rz1-x}
        |z_{n-1} - x| = |\pi(x) - x| \le C R^*(x).
    \end{equation}
    Thus, either case in \eqref{cond.Rz} with $j=1$ gives
    \begin{equation}\label{est.Rz.j=1}
        R^*(z_{n-1}) \le C R^*(x).
    \end{equation}
    Then consider $j=2$ in \eqref{cond.Rz}. By \eqref{est.Rz1-x}, \eqref{est.Rz.j=1} and \eqref{est.R*x2}, we have
    \begin{equation}
        |z_{n-2} - x| \le |z_{n-1} - x| + |z_{n-2} - z_{n-1}| \le CR^*(x) + \frac{1}{c^*} R^*(z_{n-1}) \le CR^*(x).
    \end{equation}
    Hence, \eqref{cond.Rz} with $j=2$ gives
    \begin{equation}
        R^*(z_{n-2}) \le CR^*(x).
    \end{equation}
    Repeating this argument until $j = j_s$, we get $R^*(z_{n-j_s}) \le CR^*(x)$ for some large $C>0$. But since we know by our construction that $R^*(z_{n-j_s}) \ge c'r_0$. Then we must have $R^*(x) \ge cr_0$ for some small $c>0$. This situation has been solved at the beginning of the proof.

    We only need to continue with the case when $k_1$ exists. If $k_1 = j_s$, we stop since $R^*(z_{n-k_1}) \ge c' r_0$. Otherwise, let $k_2$ with $k_1< k_2 \le j_s$ be the smallest integer such that
    \begin{equation}
        R^*(z_{n-k_2}) > 32\ell_0^{-1} |z_{n-k_2} - z_{n-k_1}| \quad \text{and} \quad R^*(z_{n-k_2}) > 32 \ell_0^{-1} R^*(z_{n-k_1}).
    \end{equation}
    Again, if $k_2$ does not exist, we can show by a similar argument that $R^*(z_{n-k_2}) \ge cr_0$. In this case, we have
    \begin{equation}
        N_u(z_{n-k_1}, \frac18 \ell_0 R^*(z_{n-k_1})) \le CN_u(0,\frac12 r_0),
    \end{equation}
    which together with \eqref{est.N.n-k1} gives \eqref{est.unifdoubling}.

    It remains to consider the case when $k_2$ exists. Applying a similar argument as before, we have
    \begin{equation}
        N_u(z_{n-k_1},  \frac18 \ell_0 R^*(z_{n-k_1})) \le C N_u(z_{n-k_2}, \frac18 \ell_0 R^*(z_{n-k_2})).
    \end{equation}
    Repeating this process until $k_{m+1}$ does not exist or $k_m = j_s$, we get a sequence of $k_i$, $1\le i\le m$, with $R^*(z_{n-k_m}) \ge cr_0$. Moreover, it holds for each $1\le i\le m-1$
    \begin{equation}
        N_u(z_{n-k_i}, \frac18 \ell_0 R^*(z_{n-k_i})) \le C N_u(z_{n-k_{i+1}}, \frac18 \ell_0 R^*(z_{n-k_{i+1}})),
    \end{equation}
    and
    \begin{equation}
        N_u(z_{n-k_m}, \frac18 \ell_0 R^*(z_{n-k_m})) \le CN_u(0,\frac12 r_0).
    \end{equation}
    Finally, the last two estimates combined with \eqref{est.N.n-k1} imply the desired estimate.
    \end{proof}

\section{Estimates of nodal sets}\label{sec.nodal}
\subsection{Nodal sets of harmonic functions } We first recall the well-known interior estimate of nodal sets over a flat boundary \cite{DF88,L91}.

\begin{lemma}\label{lem.flatbdary} 
Let $x\in \{ y = (y',y^n): y^n \ge 0  \}$.
Let $u$ be a non-zero harmonic function in $B_{8r}(x) \cap \{ y^n > 0 \}$ for some $r>0$ and $u = 0$ on $B_{8r}(x) \cap \{ y^n = 0 \}$.  There exists $C$ depending only on $n$ such that 
\begin{equation}
\mathbb{H}^{n-1}(Z(u)\cap B_r(x))\le C(N_{u}(x,4r)+1) r^{n-1}.
\end{equation}
\end{lemma}

With the above lemma, we can obtain the estimate of nodal sets for harmonic functions in a star-shaped subdomain of a polytope.

\begin{theorem}\label{thm.local}
    Under the same assumption of Lemma \ref{lem.unifDI}, we have
    \begin{equation}\label{est.ZuinBalls}
        \mathbb{H}^{n-1}(z(u) \cap P \cap B(0,\frac18 r_0)) \le C(N_u(0,\frac12 r_0) + 1) r_0^{n-1}.
    \end{equation}
\end{theorem}
\begin{proof}
    Let us define $E_{k}$ as follows:
    \begin{equation}\label{4.51}
    E_{k}=\Big\{ y\in B(0,\frac18 r_0) \cap P :2^{-k} r_0 < d(y,\mathscr{F}^{n-2}) \le 2^{-k+1}r_0 \Big\}, \quad k=1,2,3, \cdots
    \end{equation}
    Then
    \begin{equation}
        B(0,\frac18 r_0) \cap P = \bigcup_{k=1}^\infty E_k.
    \end{equation}

    Fix $k\ge 1$. By \eqref{est.R*x}, for each $x\in E_k$, $R^*(x) \ge 2^{-k} r_0$. We can cover $E_k$ by a finite number of balls $B(x_j, 2^{-k-2}r_0)$ with bounded overlaps such that $x_j \in E_k$ and
    \begin{equation}
        E_k \subset \bigcup_{j} B(x_j, c_1 2^{-k-5} r_0) \cap P,
    \end{equation}
    where $c_1$ is the constant given in Lemma \ref{lem.unifDI}. The number of these balls is bounded by $ C 2^{(n-2)k}$ since $\mathscr{F}^{n-2}$ is $(n-2)$-dimensional.
    Since $d(x_j, \mathscr{F}^{n-2}) \ge 2^{-k}r_0$, $B(x_j, 2^{-k-2} r_0)$ does not intersect with $\mathscr{F}^{n-2}$. Thus $B(x_j, 2^{-k-2} r_0) \cap \partial P$ is an empty set or a portion of a flat plane. Hence, we can apply Lemma \ref{lem.flatbdary} to get
    \begin{equation}
        \mathbb{H}^{n-1}(Z(u)\cap B(x_j, c_1 2^{-k-5} r_0) \cap P) \le C(N_u(x_j, c_1 2^{-k-3}r_0) + 1) (2^{-k-5} r_0)^{n-1}.
    \end{equation}
    Since $R^*(x_j) \ge 2^{-k} r_0$, then by Lemma \ref{lem.unifDI},
    \begin{equation}
        N_u(x_j, c_1 2^{-k-3}r_0) \le N_u(x_j, c_1 (R^*(x_j) \wedge \frac14 r_0)) \le CN_u(0, \frac12 r_0).
    \end{equation}
    It follows that
    \begin{equation}
    \begin{aligned}
        \mathbb{H}^{n-1}(Z(u)\cap E_k) & \le \sum_{j} \mathbb{H}^{n-1}( Z(u) \cap B(x_j, c_1 2^{-k-5} r_0) \cap P ) \\
        & \le C \sum_{j} ( N_u(x_j, c_1 2^{-k-3} r_0) + 1) 2^{-(k+5)(n-1)} r_0^{n-1} \\
        & \le C 2^{(n-2)k} ( N_u(0, \frac12 r_0) + 1) 2^{-(k+5)(n-1)} r_0^{n-1} \\
        & \le C (N_u(0, \frac12 r_0) + 1) 2^{-k} r_0^{n-1}.
    \end{aligned}
    \end{equation}
    Summing over $k\ge 1$, we obtain \eqref{est.ZuinBalls}.
    \end{proof}

\subsection{Dirichlet eigenfunctions}
Let $P$ be a bounded Lipschitz polytope. Let $\varphi_{\lambda}$ be the Dirichlet eigenfunction of $-\Delta$ corresponding to the eigenvalue $\lambda>0$, namely, $-\Delta\varphi_{\lambda}=\lambda \varphi_{\lambda}$ in $P$ and $\varphi_{\lambda} = 0$ on $\partial P$. Let 
\begin{equation}\label{eq.lifting}
u_{\lambda}(x,t)= e^{t\sqrt{\lambda}} \varphi_{\lambda}(x).
\end{equation}
Then $u_{\lambda}$ is harmonic in $\widetilde{P}:=P \times \mathbb{R}$ and $u_\lambda = 0$ on $\partial \widetilde{P} = \partial P \times \R$.

The next lemma gives an upper bound of the doubling index for $u_\lambda$ at large scales.

\begin{lemma} \label{lem.DLintP}
Let P be a bounded Lipschitz polytope as above. Let $u_{\lambda}$ be the harmonic extension in $\widetilde{P}$ of the Dirichlet eigenfunction $\varphi_{\lambda}$ given by \eqref{eq.lifting}. There exists $r_{1}=r_{1}(P) > 0$ such that if $(x,t) \in \widetilde{P}$ and $ 0 < r <r_{1}, -1<t<1$ , then $ N_{u_{\lambda}}((x,t),r)\le C_{r}\sqrt{\lambda}$, where $C_{r}$ depends only on r and P. 
\end{lemma}

The proof of the above lemma is standard. We refer to \cite{LMNN21} for a proof in general Lipschitz domains and to \cite{DF88,DF90} for smooth domains. The details are skipped.

Finally , we prove Theorem \ref{thm:1} .

\begin{proof}[Proof of Theorem \ref{thm:1}]
To estimate the nodal set of $\varphi_\lambda$ in $P$, it suffices to estimate the nodal set of $u_\lambda$ in $P\times (-\delta, \delta)$ for some small $\delta>0$ (independent of $\lambda$).

We need the following property: if $B(x,r) \cap P$ is star-shaped with respect to $x$, then $B((x,0),r) \cap \widetilde{P}$ is star-shaped with respect to $(x,0)$. Note that $y\in \partial P$ if and only if $(y,t) \in \partial \widetilde{P}$. 
Assume $B(x,r) \cap P$ is star-shaped with respect to $x$.
To show $B((x,0),r) \cap \widetilde{P}$ is star-shaped with respect to $(x,0)$, we consider a point $(y,s) \in B((x,0),r) \cap \partial \widetilde{P}$. Note that this implies $y\in B(x,r) \cap \partial P$. Since $B(x,r) \cap P$ is star-shaped with respect to $x$ and by the definition, for any $y\in B(x,r) \cap \partial P$, we have
\begin{equation}
    (y-x)\cdot n(y) \ge 0.
\end{equation}
Observe that the outer normal of $(y,s) \in \partial \widetilde{P}$ is $(n(y),0)$. Then it is easy to verify
\begin{equation}
    ((y,s) - (x,0)) \cdot (n(y), 0) = (y-x)\cdot n(y) \ge 0.
\end{equation}
Since $(y,s) \in B((x,0),r) \cap \partial \widetilde{P}$ is arbitrary, we see that $B((x,0),r) \cap \widetilde{P}$ is star-shaped with respect to $(x,0)$. The property is proved.

Recall the partition of $\partial P$ in Proposition \ref{prop.partition}, i.e., there exists $\{ x_{k,j} \} \subset \partial P$ such that
\begin{equation}
    \partial P \subset \bigcup_{k=0}^{n-1} \bigcup_{j=1}^{M_k} B(x_{k,j}, 2(\frac{1}{32}c^{*})^{k+1}r_0).
\end{equation}
Moreover, $B(x_{k,j}, (\frac{1}{32} c^*)^k c^* r_0) \cap P$ is star-shaped. Thus by the previous property, the extended subdomains $B((x_{k,j}, 0), (\frac{1}{32} c^*)^k c^* r_0) \cap \widetilde{P}$ are star-shaped with respect to $(x_{k,j},0)$. In addition, for some $\delta_1>0$, we have
\begin{equation}\label{eq.tP.partition}
    B(\partial P \times \{ 0\}, \delta_1 ) \cap \widetilde{P} \subset \bigcup_{k=0}^{n-1} \bigcup_{j=1}^{M_k} B((x_{k,j},0), 4(\frac{1}{32}c^{*})^{k+1}r_0) \cap \widetilde{P}.
\end{equation}
Lemma \ref{lem.DLintP} indicates that
\begin{equation}\label{est.DLxkj}
    N_{u_\lambda}((x_{k,j},0), \frac12 (\frac{1}{32} c^*)^k c^* r_0) \le C \sqrt{\lambda}.
\end{equation}
By Theorem \ref{thm.local} (with $r_0$ replaced by $32(\frac{1}{32}c^{*})^{k+1}r_0 = (\frac{1}{32}c^{*})^{k} c_* r_0 $), we have
\begin{equation}
\begin{aligned}
    & \mathbb{H}^{n}(B((x_{k,j},0), 4(\frac{1}{32}c^{*})^{k+1}r_0) \cap \widetilde{P} \cap Z(u_\lambda)) \\
    & \le C(N_{u_\lambda}((x_{k,j},0), 16(\frac{1}{32}c^{*})^{k+1}r_0 ) + 1) r_0^{n} \\
    & \le C\sqrt{\lambda} r_0^{n},
\end{aligned}
\end{equation}
where we also used \eqref{est.DLxkj} in the last inequality.
By \eqref{eq.tP.partition} and the fact $M_k$ is bounded by a constant depending on $P$ and $r_0$,
\begin{equation}
    \mathbb{H}^{n}(B(\partial P \times \{ 0\}, \delta_1 ) \cap \widetilde{P} \cap Z(u_\lambda)) \le C\sqrt{\lambda}.
\end{equation}
Let $P_\delta = B(\partial P, \delta) \cap P$.
By a simple geometric observation, there exists $\delta>0$ such that 
\begin{equation}
    P_\delta \times (-\delta,\delta) \subset B(\partial P \times \{ 0\}, \delta_1 ) \cap \widetilde{P}.
\end{equation}
Therefore,
\begin{equation}
    \mathbb{H}^{n}(P_\delta \times (-\delta,\delta) \cap Z(u_\lambda)) \le C\sqrt{\lambda}.
\end{equation}

On the other hand, since $(P\setminus P_\delta) \times (-\delta, \delta)$ is an interior region away from the boundary $\partial \widetilde{P}$, by the interior estimate of the nodal set in Lemma \ref{lem.flatbdary} and Lemma \ref{lem.DLintP} (combined with a partition of interior balls),
\begin{equation}
    \mathbb{H}^n((P\setminus P_\delta) \times (-\delta, \delta) \cap Z(u_\lambda)) \le C\sqrt{\lambda}.
\end{equation}
It follows that
\begin{equation}
    \mathbb{H}^n(P \times (-\delta, \delta) \cap Z(u_\lambda)) \le C\sqrt{\lambda}.
\end{equation}
Consequently, due to \eqref{eq.lifting},
\begin{equation}
    \mathbb{H}^{n-1}(P \cap Z(\varphi_\lambda)) = (2\delta)^{-1} \mathbb{H}^n(P \times (-\delta, \delta) \cap Z(u_\lambda)) \le C\sqrt{\lambda}.
\end{equation}
This proves the main theorem.
\end{proof}

\appendix
\section{{Outline of the proof of Lemma \ref{lem.monotonicityDL}}}

\begin{proof}
Without loss of generality, assume $x = 0\in\overline{\Omega}$.
Consider a harmonic function $u$ defined in $B_{2R}(0) \cap \Omega$ satisfying the Dirichlet boundary condition $u = 0$ on $B_{2R}(0) \cap \partial \Omega$. 

\textbf{Step 1}: Monotonicity of Almgren's frequency function in star-shaped domains.
For $r\in (0,2R)$, define
\begin{equation*}\begin{aligned}
  &H(r)=\int_{\partial B_{r}\left(0\right)\cap\Omega} u^{2} d\sigma,\\
  &D(r)=\int_{B_{r}(0)} |\nabla u|^{2},
\end{aligned}\end{equation*}
where $d\sigma$ denotes the surface measure. The Almgren's frequency function of $u$ centered at $0$ is given by
\begin{equation}\label{def.frequency}
\beta_{u}(r) = \beta_u(0,r) =\frac{rD(r)}{H(r)}.
\end{equation}
It can be shown that $\beta_u(r)$ is nondecreasing in $r\in (0,2R)$, provided that $B_{2R}(0)\cap \Omega$ is star-shaped with respect to the center $0$. A rigorous proof can be found in \cite{KN98}.

\textbf{Step 2}: Monotonicity of the doubling index.
The relation between $H(r)$ and the frequency function is given by
\begin{equation}\label{eq.H}
\frac{H'(r)}{H(r)}-\frac{n-1}{r}=\frac{2}{r}\beta_{u}(r).
\end{equation}
Integrating (\ref{eq.H}) over $r\in(r_{1},r_{2}) \subset (0,2R)$, we obtain
\begin{equation}\label{eq.Hr2/Hr1}
 \log \frac{H(r_{2})}{H(r_{1})}-(n-1)\log\frac{r_{2}}{r_{1}}=2\int_{r_{1}}^{r_{2}}\frac{\beta_{u}(r)}{r}dr \end{equation}
Taking $r_1 = t$ and $r_2 = 2t$ in the above identity and using the monotonicity of the frequency function in Step 1, we derive from \eqref{eq.Hr2/Hr1} that $H(2t)/H(t)$ is nondecreasing in $t\in (0,R)$. Therefore, for any $0<\tau<t<R$,
\begin{equation}
    \frac{H(2\tau)}{H(\tau)} \le \frac{H(2t)}{H(t)}.
\end{equation}
Equivalently,
\begin{equation}\label{est.4sphere}
    H(2\tau) H(t) \le H(\tau) H(2t).
\end{equation}
Fix $s$ and $r$ such that $0<s<r<R$ and let $\tau \in (0,s)$ and $t\in (s,r)$. Integrating \eqref{est.4sphere} in $\tau \in (0,s)$ and in $t\in (s,r)$, we have
\begin{equation}
    \int_{B_{2s} \cap \Omega} u^2 \int_{B_r\setminus B_s  \cap \Omega} u^2 \le \int_{B_s  \cap \Omega} u^2 \int_{B_{2r} \setminus B_{2s}  \cap \Omega} u^2.
\end{equation}
Adding $\int_{B_{2s}  \cap \Omega} u^2 \int_{B_s  \cap \Omega} u^2$ on both sides, we get
\begin{equation}
    \int_{B_{2s}  \cap \Omega} u^2 \int_{B_r  \cap \Omega} u^2 \le \int_{B_s  \cap \Omega} u^2 \int_{B_{2r}  \cap \Omega} u^2.
\end{equation}
This implies \eqref{est.monotonicityDL} as desired.
\end{proof}

\bibliographystyle{abbrv}
\bibliography{Myref}

\begin{thebibliography}{10}

\bibitem{DF88}
H.~Donnelly and C.~Fefferman.
\newblock Nodal sets of eigenfunctions on {R}iemannian manifolds.
\newblock {\em Invent. Math.}, 93(1):161--183, 1988.

\bibitem{DF90}
H.~Donnelly and C.~Fefferman.
\newblock Nodal sets of eigenfunctions: {R}iemannian manifolds with boundary.
\newblock In {\em Analysis, et cetera}, pages 251--262. Academic Press, Boston,
  MA, 1990.

\bibitem{G22}
J.~M. Gallegos.
\newblock Size of the zero set of solutions of elliptic {PDE}s near the
  boundary of {L}ipschitz domains with small {L}ipschitz constant.
\newblock {\em Calc. Var. Partial Differential Equations}, 62(4):Paper No. 113,
  52, 2023.

\bibitem{H23}
H.~Hezari.
\newblock Upper bounds on the size of nodal sets for {G}evrey and quasianalytic
  {R}iemannian manifolds.
\newblock {\em Comm. Math. Phys.}, 404(3):1341--1359, 2023.

\bibitem{KN98}
I.~Kukavica and K.~Nystr\"{o}m.
\newblock Unique continuation on the boundary for {D}ini domains.
\newblock {\em Proc. Amer. Math. Soc.}, 126(2):441--446, 1998.

\bibitem{L91}
F.~Lin.
\newblock Nodal sets of solutions of elliptic and parabolic equations.
\newblock {\em Comm. Pure Appl. Math.}, 44(3):287--308, 1991.

\bibitem{LZ22}
F.~Lin and J.~Zhu.
\newblock Upper bounds of nodal sets for eigenfunctions of eigenvalue problems.
\newblock {\em Math. Ann.}, 382(3-4):1957--1984, 2022.

\bibitem{L18b}
A.~Logunov.
\newblock Nodal sets of {L}aplace eigenfunctions: polynomial upper estimates of
  the {H}ausdorff measure.
\newblock {\em Ann. of Math. (2)}, 187(1):221--239, 2018.

\bibitem{L18a}
A.~Logunov.
\newblock Nodal sets of {L}aplace eigenfunctions: proof of {N}adirashvili's
  conjecture and of the lower bound in {Y}au's conjecture.
\newblock {\em Ann. of Math. (2)}, 187(1):241--262, 2018.

\bibitem{LM18}
A.~Logunov and E.~Malinnikova.
\newblock Nodal sets of {L}aplace eigenfunctions: estimates of the {H}ausdorff
  measure in dimensions two and three.
\newblock In {\em 50 years with {H}ardy spaces}, volume 261 of {\em Oper.
  Theory Adv. Appl.}, pages 333--344. Birkh\"{a}user/Springer, Cham, 2018.

\bibitem{LM20}
A.~Logunov and E.~Malinnikova.
\newblock Review of {Y}au's conjecture on zero sets of {L}aplace
  eigenfunctions.
\newblock In {\em Current developments in mathematics 2018}, pages 179--212.
  Int. Press, Somerville, MA, [2020] \copyright 2020.

\bibitem{LMNN21}
A.~Logunov, E.~Malinnikova, N.~Nadirashvili, and F.~Nazarov.
\newblock The sharp upper bound for the area of the nodal sets of {D}irichlet
  {L}aplace eigenfunctions.
\newblock {\em Geom. Funct. Anal.}, 31(5):1219--1244, 2021.

\bibitem{Z19}
J.~Zhu.
\newblock Doubling inequality and nodal sets for solutions of bi-{L}aplace
  equations.
\newblock {\em Arch. Ration. Mech. Anal.}, 232(3):1543--1595, 2019.

\bibitem{ZZ23}
J.~Zhu and J.~Zhuge.
\newblock {N}odal sets of {D}irichlet eigenfunctions in quasiconvex {L}ipschitz
  domains.
\newblock {\em arXiv:2303.02046}, 2023.

\end{thebibliography}
\end{document}